\documentclass[ejsv2,preprint]{imsart}

\RequirePackage[numbers]{natbib}
\usepackage{titlesec}
\usepackage{orcidlink}
\usepackage{comment}
\usepackage{appendix}
\newtheorem{prop}{Proposition}[section]
\newtheorem*{prop*}{Proposition}
\newtheorem{corollary}{Corollary}[section]

\newtheorem{lemma}{Lemma}[section]

\newcommand{\Expect}[1]{\mathbb{E}\left[ #1 \right]}

\newcommand{\diffd}{\textnormal{d}}

\newcommand{\inprod}[2]{\left\langle #1,#2\right\rangle}

\newcommand{\norm}[1]{\left|\left|  #1 \right| \right|}

\newcommand{\SPAN}[1]{\text{span}\left(#1\right)}

\newcommand{\weakto}{\rightharpoonup}

\begin{document}
\begin{frontmatter}
\title{An RKHS Framework for Fixed Effects in Permanental Process Models}

\runtitle{Permanental Processes with Fixed Effects}

\begin{aug}
\author[A]{\fnms{Matthew}~\snm{LeDuc}\ead[label=e1]{matthew.leduc@colorado.edu}\orcid{0009-0006-5892-8823}}

\address[A]{Department of Applied Mathematics,
University of Colorado Boulder\printead[presep={,\ }]{e1}}
\runauthor{M. LeDuc }
\end{aug}

\begin{abstract}
This short work describes an extension of the permanental process model which includes fixed effects. By starting with a prior on the fixed effects coefficients we show that, in the diffuse prior limit, the intensity function of the permanental process can be found using the representer theorem and naturally decomposed into a fixed effects term and a function which is an element of a Reproducing Kernel Hilbert Space (RKHS). We show that the limiting equivalent kernel defines an RKHS whose squared norm is exactly the limiting penalty. This allows for straightforward scientific interpretation of permanental process models and the easy incorporation of domain knowledge into the estimation process. 
\end{abstract}


\begin{keyword}
\kwd{Point processes}
\kwd{Reproducing Kernel Hilbert Spaces}
\end{keyword}

\end{frontmatter}

\section{Introduction}

Many common datasets, such as the location of trees of a certain species, the failure time of a set of mechanical parts, the locations of fish catches, or the detections of remotely sensed emissions, can be described as a random set of spatial or temporal locations at which specific events occurred. These datasets are commonly modeled as a realization of a Poisson Point Process \cite{streit.ppp}, which posits that for some latent function $\lambda(s)$, the number of events $N$ in a measurable set $D$ is Poisson distributed with mean 
\begin{equation}
    \Expect{N} = \int_D\lambda(s)\diffd\mu(s)
\end{equation}
where $\mu(s)$ is the ambient measure on $D$. In many applications, $\lambda(s)$ varies substantially over space and time, making highly flexible models desirable. One widely used class of models, known as Cox process models, generalizes $\lambda(s)$ to be a random function \cite{coxprocesses}. Of these, perhaps the most commonly used is the log-Gaussian Cox Process, which models $\lambda(s)$ in the form
\begin{equation}
    \lambda(s)=\exp\left(f(s)\right)
\end{equation}
for some Gaussian process $f(s)$ \cite{moller_lgcp}. These models are popular for their flexibility and interpretability, however inference can be computationally demanding, particularly for large datasets or finely discretized spatial domains. This motivates the development of other Cox process models that retain the flexibility of the log-Gaussian Cox process while remaining computationally tractable.

The permanental process \cite{mccullagh_permproc} is a particular case of the Cox process, where the intensity function $\lambda(s)$ has the form
\begin{equation}
    \begin{split}
        \lambda(s) = &\frac{1}{2}\sum_{j=1}^Nf_j(s)^2\\
        f_j(s)\sim& \mathcal{GP}(0,k(s,t))
    \end{split}
\end{equation}
Its name is due to the fact that the order-n product density is given by 
\cite{mccullagh_permproc}
\begin{equation}
    \rho^{(n)}(s_1,...,s_n)=\Expect{\prod_{i=1}^n\lambda(s_i)} = \text{per}_{\alpha}(\mathbf{K})
\end{equation}
where $\mathbf{K}_{ij}=k(s_i,s_j)$, $\alpha=n/2$, and $\text{per}_{\alpha}$ is the $\alpha$-weighted permanent:
\begin{equation}
    \text{per}_{\alpha}(\mathbf{K})=\sum_{\sigma(1,...,n)}\alpha^{\#\sigma}\prod_{i=1}^nK(s_i,s_{\sigma(i)})
\end{equation}
where $\sigma(1,...,n)$ is the set of permutations of $1,...,n$ and $\#\sigma$ is the number of cycles in the permutation. They also exhibit point clustering, making them a natural model for processes such as the location of trees in a forest that exhibit this behavior \cite{moller2005properties}.

Models based on the permanental process have found application in a wide variety of areas such as survival analysis, remote sensing, modeling of point patterns on networks,  spatial transportation data, wildfire modeling, and classification \cite{kim2023survivalpermprocs,PoissonRatioUQ,mine_permproc,coxprocessesnetworks,sellier2023sparse,thongtha2023normal,yang2012permprocclassification}. This is in part because they are particularly favorable computationally due to the ability to write the solution as a member of a Reproducing Kernel Hilbert Space (RKHS) and leverage the representer theorem to write the latent function $f$ at an arbitrary location as a weighted sum $f(s)=\sum\alpha_ik(s_i,s)$, leading to a finite-dimensional optimization problem \cite{pmlr-v54-flaxman17a,pmlr-v70-walder17a}. 
While the structure of the permanental process model as described assumes that the latent process $f(s)$ has no mean trend, in many point processes, such as the locations of trees or the location of catches of certain species of fish, the intensity will vary as a function of, for example, soil quality, elevation, salinity, or sea surface temperature. It is reasonable, then, to question whether a model of the form
\begin{equation}\label{eq:lambda_fixedeffects}
    \begin{split}
        \lambda(s)=&\frac{c}{2}f(s)^2 \\
        f(s)=&X(s)\beta+g(s)\\
        g(s)\sim&\mathcal{GP}(0,k(s,t))
    \end{split}
\end{equation}
can be fit using a similar procedure. In this paper, we demonstrate that the assumed kernel can be modified in a consistent way to allow the estimation of the intensity function in the same manner as \cite{pmlr-v54-flaxman17a} via the representer theorem, but without applying the RKHS penalty to the directions along the columns of $X$. Models of the related form $f(s)=\beta+g(s)$ have been proposed to solve the nodal line problem, where a zero-crossing of the process $g(s)$ yields a nonphysical crease in the intensity surface \cite{kim2022fast,sun2024nonstationarysparsespectralpermanental}.

This has many natural comparisons, for example smoothing spline estimation wherein elements of a finite-dimensional null-space of the penalty, for example terms of the form $ax+b$, are left free and a smoothness penalty is imposed on the remainder \cite{kimeldorfwahba1970correspondence,wahba1990spline}. Here, in the diffuse prior limit the columns of $X$ play a similar role in denoting the directions along which we do not want to apply the RKHS penalty, while the remainder is regularized according to a covariance kernel $k(s,t)$. We will show that, under relatively mild assumptions on $X$, namely that the fixed effects $x_j(s)$ are continuous and linearly independent functions of space, the diffuse-prior limit defines a valid RKHS in which the representer theorem continues to apply, and also yields a straightforward decomposition allowing the recovery of the fixed effect coefficients. Notably, this is possible even if the fixed effect functions $x_j(s)$ are not elements of the RKHS $\mathcal{H}_K$.


\section{RKHS Formulation and the Diffuse-Prior Limit}

In this section, we will derive the RKHS formulation and show that, in the diffuse prior limit as $\tau\to\infty$, we are able to write the solution as an element of a RKHS and convert the problem into a finite-dimensional optimization problem. We will begin by listing the assumptions that are necessary to do this, and then describe the limiting behavior of both the equivalent kernel and penalty and show that they correspond, allowing application of the representer theorem to the solution of the problem \cite{gen_representer_thm}.
\subsection{Assumptions}

Before we begin, we must describe the assumptions of this model. These are generally not overly restrictive, however each of them are necessary for this model.

First, as done in \cite{pmlr-v54-flaxman17a} and \cite{pmlr-v70-walder17a}, we will assume that the constants $c,\gamma$ are both positive, real constants. Now we assume that the point process is observed on $\Omega$ a compact domain equipped with a finite measure $\mu$, and let
$\mathcal{L}^2(\Omega)=\mathcal{L}^2(\Omega,\mu)$. Suppose that the kernel $k:\Omega\times\Omega\to\mathbb{R}$ is continuous, symmetric, and positive-definite with associated integral operator 
\begin{equation}
        Kf(s)
        =
        \int_{\Omega} k(s,t)f(t)\,\diffd\mu(t).
\end{equation}
This gives us that the operator $K:\mathcal{L}^2(\Omega)\to C(\Omega)$ is bounded, self-adjoint, and positive. Throughout, we will refer to the kernel function with lowercase letters and the integral operator associated with the kernel using uppercase. 

Lastly, we assume that the fixed-effects operator $X$ has the form
\begin{equation}
    X\beta=\sum_{j=1}^p\beta_jx_j(s)
\end{equation}
with each $x_j\in C(\Omega)$ linearly independent as elements of $\mathcal{L}^2(\Omega)$, equivalently the operator $X:\mathbb{R}^p\to C(\Omega)$ is injective. Notably, we do not require that 
$\operatorname{span}\{x_1,\ldots,x_p\}
        \subset \mathcal{H}_K$ or that 
        \\ $\operatorname{span}\{x_1,\ldots,x_p\}
        \bigcap \mathcal{H}_K=\{0\}$.

\subsection{The Kernel for finite $\tau$} 
We consider the problem of estimating the latent intensity of a point process $P$ with intensity function $\lambda(s)$ having the form in Eq. \eqref{eq:lambda_fixedeffects}. By incorporating the prior $\beta\sim N(0,\tau^2I)$, we have the penalized log-likelihood
\begin{equation}\label{eq:k_kernel_tau}
    \ell(f)=\sum_{i=1}^d\log\left(\frac{c}{2}f(s_i)^2\right) - \frac{c}{2}\int_{\Omega}f(s)^2 \diffd s-\frac{\gamma}{2}\|g\|_{\mathcal{H}_K}^2-\frac{1}{\tau^2}\|\beta\|_2^2
\end{equation}
By integrating $\beta$ out of the prior and noting that $g(s)=f(s)-X(s)\beta$, we can obtain the equivalent formulation
\begin{equation}
    \label{eq:permproclike_marginalized}
    \begin{split}
    \ell(f) = &\sum_{i=1}^d\log\left(\frac{c}{2}f(s_i)^2\right) - \frac{c}{2}\int_{\Omega}f(s)^2 \diffd s-\frac{\gamma}{2}\|f\|_{\mathcal{H}_{K_\tau}}^2 \\
k_\tau(s,t) = & k(s,t)+\tau^2 X(s)X(t)^*
    \end{split}
\end{equation}
Since $X:\mathbb{R}^p\to C(\Omega)$ is a finite rank linear operator it is bounded, and the kernel $\tau^2 X(s)X(t)^*$ defines a finite-dimensional RKHS given by $\SPAN{ x_1(s),...,x_p(s) }$. Thus, $k_{\tau}(s,t)$ defines its own RKHS $\mathcal{H}_{K_\tau}$ with norm \cite{rkhs_book}
\begin{equation}\label{eq:Ktau_norm}
\|f\|_{K_\tau}^2 =\underset{f=g+X\beta,g\in \mathcal{H}_K}{\inf}\left(  \|g\|_{K}^2+\frac{1}{\tau^2}\|\beta\|_2^2\right)
\end{equation}

Since the integral term in Eq. \eqref{eq:permproclike_marginalized} is $\|f\|_2^2$, we can represent both norms as a single quadratic form of $f$ as 
\begin{equation}
    \frac{c}{2}\|f\|_2^2+\frac{\gamma}{2}\|f\|_{K_\tau}^2=\frac{1}{2}\inprod{f}{\left(cI+\gamma K_\tau^{-1}\right) f}=\frac{1}{2}\inprod{f}{R_\tau^{-1} f}
\end{equation}
As shown in \cite{pmlr-v54-flaxman17a}, for finite $\tau$ the kernel $R_{\tau}= K_\tau\left(cK_\tau+\gamma I\right)^{-1}$ defines its own RKHS which is norm-equivalent to that induced by $K_\tau$. However, the behavior as $\tau\to \infty$ is more complex, and the subject of the following section.

\subsection{Limiting Behavior of the Kernel}

The parameter $\tau$ represents the prior standard deviation of the parameters $\beta$, and as $\tau\to\infty$ the prior on these coefficients becomes more and more diffuse and the penalty on the directions aligned with $X(s)$ vanishes. It is natural, therefore, to wonder whether the operators $R_\tau$ converge to a well-defined limit and, if so, to analyze the properties of this limit. 

To proceed, we examine the form of $R_\tau$ for finite $\tau$. It can be seen that, letting $A=cK+\gamma I$, $R_\tau=\frac{1}{c}\left(I -\gamma\left(A^{-1}-A^{-1}X\left(\frac{1}{c\tau^2}+X^*A^{-1}X\right)^{-1}X^*A^{-1}\right)\right)$, motivating the following lemma.
\begin{lemma}\label{lemma:Rinfty}
    $R_\tau$ converges uniformly to the operator $R_{\infty}$ given by
    \begin{equation}
       R_\infty=\frac{1}{c}\left(I -\gamma\left(A^{-1}-A^{-1}X\left(X^*A^{-1}X\right)^{-1}X^*A^{-1}\right)\right) 
    \end{equation}
    as $\tau\to\infty$.
\end{lemma}
\begin{proof}
We wish to show that, as $\tau\to\infty$, $\|R_\tau-R_\infty\|\to 0$. Letting $S=X^*A^{-1}X$, we can say that
\begin{equation}
\begin{split}
    \|R_\tau-R_\infty\| = &\frac{\gamma}{c}\norm{ A^{-1}X\left(\left(\frac{1}{c\tau^2}I+S\right)^{-1}-S^{-1}\right)X^*A^{-1}}\\
    \le & \frac{\gamma}{c}\norm{X^*A^{-1}}^2\norm{\left(\frac{1}{c\tau^2}I+S\right)^{-1}-S^{-1}}\\
    = &\frac{\gamma}{c^2\tau^2}\norm{X^*A^{-1}}^2 \norm{\left(\frac{1}{c\tau^2}I+S\right)^{-1}S^{-1}}
    \end{split}
\end{equation}
Now $S$ is invertible, and by the spectral theorem can be diagonalized, so we have that
\begin{equation}
    \|R_\tau-R_\infty\| \le \frac{\gamma}{c^2\tau^2}\norm{X^*A^{-1}}^2 \underset{j}{\max}\frac{1}{\lambda_j(\lambda_j+(c\tau^2)^{-1})}
\end{equation}
which approaches zero as $\tau\to\infty$, as desired. Thus $R_\tau \to R_\infty$ uniformly.
\end{proof}
Now we wish to demonstrate that $R_\infty$ defines an RKHS on its own. This is the objective of the following proposition:
\begin{prop}\label{prop:Rinfty_defines_rkhs}
    Under the assumptions so far described, the operator $R_{\infty}$ is the kernel operator of an RKHS $\mathcal H_{R_\infty}$.
\end{prop}
\begin{proof}
    Since $K_\tau$ and $(cK_\tau+\gamma I)^{-1}$ commute for finite $\tau$, $R_\infty$ is a uniform limit of bounded, positive, and self-adjoint operators by Lemma \ref{lemma:Rinfty} and thus inherits those properties. Now we need to demonstrate that the kernel function $r_\infty$ is continuous on $\Omega\times \Omega$ and thus is a reproducing kernel by the Moore–Aronszajn theorem \cite{rkhs_book}. The kernel function $r_\infty$ is given by
\begin{equation}
    \label{eq:rinfty_kernelfn}
    \begin{split}
        r_\infty(s,t) =& r_0(s,t)+\frac{\gamma}{c}\sum_{i,j}S^{-1}_{i,j}a_i(s)a_j(t)\\
        a_i(s) =& A^{-1}x_i(s) 
    \end{split}
\end{equation}
where $r_0(s,t)$ is the kernel function corresponding to $\tau=0$, the same kernel function denoted the equivalent kernel in \cite{pmlr-v54-flaxman17a}. So, in essence, the addition of fixed effects introduces a finite-rank perturbation to this equivalent kernel. Since $r_0(s,t)$ is continuous, continuity of $r_\infty(s,t)$ is equivalent to showing that the finite-rank perturbation is continuous, and therefore that the kernel $r_\infty(s,t)$ is continuous. Combined with the properties of $R_\infty$ derived from uniform convergence of the operators, this yields the desired result. 

To prove this,we can show that the functions $a_i(s)$ are continuous, and then the finite rank term is continuous. Since the $x_i(s)$ are continuous, this is equivalent to proving that $A^{-1}:C(\Omega)\to C(\Omega)$. Proceed by contradiction: Suppose that there is some $h(s)\in \mathcal{L}^2(\Omega)$ that is the image of a continuous function $g(s)$ under $A^{-1}$, but is not continuous itself. Then $g(s)=Ah(s)=cK h(s)+\gamma h(s)$. But then $h(s)$ must be continuous as $K:\mathcal{L}^2(\Omega)\to C(\Omega)$ and $g(s)$ is continuous. Therefore, by contradiction $A^{-1}$ maps $C(\Omega)$ to $C(\Omega)$. Since we have assumed that all $x_i(s)$ are continuous, this implies that the $a_i(s)$ are continuous, and thus that $r_\infty$ is a continuous and positive definite kernel function on $\Omega\times \Omega$. Thus, it is a reproducing kernel and $R_\infty$ defines an RKHS.
\end{proof}
This tells us that the diffuse prior limit $\tau\to\infty$ defines an RKHS with kernel given by $r_{\infty}$. However, it remains to be shown that the limiting penalty as $\tau \to \infty$ corresponds to the norm in this same space. The limiting penalty is obtained as the limit of a family of quadratic penalties, while Proposition \ref{prop:Rinfty_defines_rkhs} establishes only the existence of the limiting RKHS and not that the limit of these penalties corresponds to the RKHS. This is a natural extension of the fact that, for each $\tau$, $R_\tau$ and $K_\tau$ are operators which define equivalent norms on the same function space, which is a significant motivation for the application of the work in \cite{pmlr-v54-flaxman17a}: by modifying the problem in this manner, we do not distort the properties of the kernel which motivated its choice to begin with.

\subsection{Limiting Behavior of the Penalty}

To obtain an analogous result to the equivalence between the norms induced by $R_\tau$ and $K_\tau$ for finite $\tau$, we must determine the limiting form of the penalty term as $\tau\to\infty$ and show that it corresponds to the norm in the RKHS induced by $R_\infty$. To do this, consider the sequence of penalties indexed by $\tau$ as 
\begin{equation}
    \label{eq:Qtau}
    Q_\tau(f) = \frac{c}{2}\|f\|_2^2+\frac{\gamma}{2}\|f\|_{K_\tau}^2 = \frac{c}{2}\|f\|_2^2+\frac{\gamma}{2}q_\tau(f)
\end{equation}
with $\|f\|_{K_\tau}$ given by Eq. \eqref{eq:Ktau_norm}. The pointwise limit of this penalty is given by
\begin{equation}
    \label{eq:Qinfty}
    Q_{\infty}(f) =\frac{c}{2}\|f\|_2^2+\frac{\gamma}{2}\underset{f=g+X\beta}{\inf}\|g\|_{K}^2=\frac{c}{2}\|f\|_2^2+\frac{\gamma}{2}q_\infty(f)  
\end{equation}
which is the square of a proper norm due to the $2-$norm term. Note that $Q_\infty(f)$ is finite on $\mathcal{H}_K\bigcup\SPAN{x_1,...,x_p} $ and infinite otherwise. The proof is as follows: For any $\tau$,$q_\tau(f)\ge q_\infty(f)$, and for any decomposition $f=g+X\beta$, we have that $q_\tau(f)\le  \|g\|_{K}^2+\frac{1}{\tau^2}\|\beta\|_2^2$. Taking $\tau\to \infty$ and then taking the infimum yields the pointwise limit. So, we have that, as $\tau\to\infty$,
\begin{equation}
    \begin{split}
        Q_\tau(f)\to& Q_{\infty}(f)\\
        R_\tau \to& R_\infty
    \end{split}
\end{equation}
and now we must associate these with each other. To accomplish this, we must first establish some properties of $Q_\infty$ as a quadratic form, notably that it is convex and lower-semicontinuous.
\begin{lemma}\label{lem:Qinfty}
    The penalty functional $Q_\infty(f)=\frac{c}{2}\|f\|_2^2+\frac{\gamma}{2}q_\infty(f)  $ is convex and lower-semicontinuous on $\mathcal{L}^2(\Omega)$.
\end{lemma}
\begin{proof}
    As the sum of convex functionals, $Q_\infty$ is itself convex. Now we proceed to showing it is lower semicontinuous. Since $\|f\|_2^2$ is a continuous functional, we must show this for $q_\infty(f)$. Suppose that $f_n\to f\in \mathcal{L}^2(\Omega)$ and let $L=\underset{n\to\infty}{\liminf}Q_{\infty}(f_n)$. In the case $L=\infty$, clearly $Q_\infty(f)\le L$. Now suppose $L$ is finite. Now for each $n$, since $q_\infty$ is an infimum over acceptable decompositions of $f$ there is a $g_n,\beta_n$ such that $f_n=g_n+X\beta_n$ and $\|g_n\|_{K}^2 \le q_\infty(f_n)+\frac{1}{n} $. Thus the supremum over $n$ of $\|g_n\|_K$ is finite, and since $\mathcal{H}_K(\Omega)$ is continuously embedded in $\mathcal{L}^2(\Omega)$, there is some $C_k$ such that $\|g_n\|_2^2\le C_k\|g_n\|_{K}^2$, so $\|g_n\|_2<\infty$. Since $g_n$ is bounded, it has a weakly convergent subsequence $g_{n'}\rightharpoonup \hat{g}$. Likewise since $f_n\to f$, $\|X\beta_n\|_2\le \|f_n\|_2^2+\|g_n\|_2^2$, which are both bounded sequences, and since $X$ is an injective finite-rank operator $\|\beta_n\|<\infty$ as well. Thus the subsequence $\beta_{n'}$ has a convergent subsequence $\beta_{n''}\to \hat{\beta}$. Thus, there is a subsequence $f_{n''}\weakto f''=\hat{g}+X\hat{\beta}$. But since weak limits are unique, $f''=f$. Now we know that, by weak convergence, $\norm{g}_K\le \liminf\|g_n\|_K$, which is bounded. 
    Since the $g_n$ were chosen to satisfy $\|g_n\|_{K}^2
\leq q_\infty(f_n)+\frac{1}{n}$ we obtain $q_\infty(f)\leq\liminf_{n\to\infty}q_\infty(f_n)$.

Therefore $q_\infty$ is lower semicontinuous. Since $\|f\|_2^2$ is continuous, $Q_\infty$ is lower semicontinuous as well.
    and so the penalty is lower semicontinuous.
\end{proof}

So, we have shown that the penalties $Q_\tau$ converge to a well-defined penalty as $\tau\to\infty$, and we have additionally shown that the limiting penalty is convex and is lower-semicontinuous on $\mathcal{L}^2(\Omega)$. The main proposition of the paper, presented here, is that this penalty is related in the expected way to the RKHS $\mathcal{H}_{R_\infty}$. This provides an analogous result to the one for finite $\tau$ and the result given in \cite{pmlr-v54-flaxman17a} about the equivalence of norms between the original and equivalent kernels.

\begin{prop}
    The penalty term $Q_\infty$ is equal to one-half of the squared norm on the RKHS $\mathcal{H}_{R_\infty}$.
\end{prop}
\begin{proof}
To begin the proof, we note that since Lemma \ref{lem:Qinfty} says that $Q_\infty$ is proper, convex, and lower semicontinuous, it implies that the penalty term $Q_\infty(f)$ is its own convex biconjugate by the Fenchel-Moreau Theorem \cite{correa2023fundamentals}. Our goal, then, is to show that the convex biconjugate of $Q_\infty$ is given by $\frac{1}{2}\norm{f}^2_{R_\infty}$.

We can write the seminorm term $q_\infty(f)$ as
\begin{equation}
    q_\infty(f)=\underset{\beta\in \mathbb{R}^p}{\inf}\begin{cases}
        \norm{f-X\beta}_K^2&f-X\beta\in \mathcal{H}_K\\
        +\infty&else
    \end{cases}
\end{equation}
which is an infimal convolution of the two functionals
\begin{equation}
    \begin{split}
        \phi(g) =&\begin{cases}
        \norm{g}_K^2&g\in \mathcal{H}_K\\
        +\infty&else
    \end{cases} \\
    \chi(h) =& \begin{cases}0&h\in \SPAN{x_1,...,x_p}\\
    +\infty & else \\ \end{cases}
    \end{split}
\end{equation}
so we can write the convex conjugate of $q_\infty$ as the sum of the convex conjugates of $\phi$ and $\chi$. These can be calculated as 
\begin{equation}
    \begin{split}
        \phi^*(h) & = \underset{g\in \mathcal{H}_K}{\sup}\left( \inprod{h}{g}-\|g\|_K^2\right) \\ 
        =& \underset{g\in \mathcal{H}_K}{\sup}\left( \inprod{Kh}{g}_K-\|g\|_K^2 \right)\\
        = &\frac{1}{4}\|Kh\|_K^2= \frac{1}{4}\inprod{h}{Kh}
    \end{split}
\end{equation}
and 
\begin{equation}
        \chi^*(h) = \begin{cases}
            0&h \perp \SPAN{x_1,...,x_p}\\
            +\infty &else
        \end{cases}
\end{equation}
so the convex conjugate of $q_\infty$ is given by
\begin{equation}
    q^*_{\infty}(h) = \begin{cases}
        \frac{1}{4}\inprod{h}{Kh} & h\perp \SPAN{x_1,...,x_p}\\
        +\infty & else
    \end{cases}
\end{equation}
Now the convex conjugate of $\norm{f}_2^2$ is given by $\frac{1}{4}\norm{h}_2^2$, and as $Q_\infty(f)=\frac{c}{2}\norm{f}_2^2+\frac{\gamma}{2}q_\infty(f)$ its convex conjugate is the infimal convolution of the convex conjugates of these terms:
\begin{equation}
    Q^*_{\infty}(h) = \underset{v\in \mathcal{L}^2(\Omega), X^*v=0}{\inf}\left[ \frac{1}{2c}\norm{h-v}_2^2 + \frac{1}{2\gamma}\inprod{v}{Kv}\right]
\end{equation}
To obtain a closed form, fix h and minimize
\begin{equation}
    \label{eq:minFhv}
    F_h(v) = \frac{1}{2c}\norm{h-v}_2^2 + \frac{1}{2\gamma}\inprod{v}{Kv}
\end{equation}
subject to the constraint that $X^*v=0$, equivalently $v\perp \text{Range}(X)$. This can be done via Lagrange multipliers by solving
\begin{equation}\label{eq:v_soln}
\begin{split}
    \frac{1}{c}(v-h)&+\frac{1}{\gamma}Kv+X\psi = 0\implies \\
    v =& A^{-1}\left(\gamma h-c\gamma X\psi\right)
    \end{split}
\end{equation}
Now we impose the condition $X^*v=0$, so 
\begin{equation}
    X^*A^{-1}\left(\gamma h-c\gamma X\psi\right)=0 \implies \psi = \frac{1}{c}\left(X^*A^{-1}X\right)^{-1}X^*A^{-1}h
\end{equation}
which tells us that the optimizer of $F_h$ is given by
\begin{equation}
    \label{eq:vmin}
    v=\gamma\left(A^{-1}-A^{-1}X(X^*A^{-1}X)^{-1}X^*A^{-1}\right)h
\end{equation}
To simplify $Q^*_\infty$ further, we can expand the 2-norm term and see that
\begin{equation}
    Q^*_\infty(h) = \underset{v\in \mathcal{L}^2(\Omega), X^*v=0}{\inf} \left[\frac{1}{2c}\|h\|_2^2-\frac{1}{c}\inprod{h}{v}+\frac{1}{2c\gamma}\inprod{v}{Av}\right]
\end{equation}
Writing $Av=\gamma h-c\gamma X\psi$ and taking the inner product with $v$, noting that $\inprod{v}{X\psi}=0$, we get that
\begin{equation}
\begin{split}
    Q^*_\infty(h) =& \frac{1}{2c}\|h\|_2^2-\frac{1}{2}\inprod{h}{v}+\frac{1}{2c}\inprod{h}{v} \\
    =&\frac{1}{2}\inprod{h}{\frac{1}{c}
    \left(I-\gamma\left(A^{-1}-A^{-1}X(X^*A^{-1}X)^{-1}X^*A^{-1}\right)\right)h} \\
    =&\frac{1}{2}\inprod{h}{R_\infty h}
    \end{split}
\end{equation}
Now the convex biconjugate of $Q_\infty$ is thus given by the extended-value quadratic form
\begin{equation}
    Q^{**}_\infty(f) =\ \frac{1}{2}\inprod{f}{R_{\infty}^{-1}f} =\frac{1}{2}\|f\|^2_{R_\infty}
\end{equation}
Now, by the Fenchel-Moreau Theorem, $Q^{**}_\infty=Q_\infty$. Thus, the limiting penalty $Q_\infty$ is exactly half the penalty derived from the norm in the RKHS $\mathcal{H}_{R_\infty}(\Omega)$, as desired.
\end{proof}

This is the central result of the paper. From this we can see that, not only can we derive an equivalent kernel formulation for the problem for finite $\tau$ in which the kernels $R_\tau$ and $K_\tau$ describe equivalent function spaces, but in the limit $\tau\to\infty$ we have a similar result: The penalty term converges to one that is equivalent to a norm on the RKHS with kernel operator $R_\infty$. From this, we now know that, in the limit $\tau\to\infty$, we can write the problem given in Eq. \eqref{eq:k_kernel_tau} as
\begin{equation}\label{eq:soln_tau_infty_equiv}
    \ell(f) = \sum_{i=1}^d\log\left(\frac{c}{2}f(s_i)^2\right)-\frac{1}{2}\|f\|_{R_\infty}^2
\end{equation}
which has the form of an empirical risk functional $\sum_{i=1}^d\log\left(\frac{c}{2}f(s_i)^2\right)$ and a regularization term $\frac{1}{2}\|f\|^2_{R_\infty}$. This immediately suggests the following corollary, which says that the solution to this problem has a finite-dimensional representation:
\begin{corollary}\label{cor:representer_cor}
    The maximizer of the likelihood functional in Equation \eqref{eq:soln_tau_infty_equiv} has the form 
    \begin{equation}\label{eq:repthm_minimizer}
    f(s)=\sum_{i=1}^d\alpha_ir_\infty(s,s_i)
    \end{equation}
    where $r_\infty(s,t)$ is the kernel function in Eq. \eqref{eq:rinfty_kernelfn}.
\end{corollary}
\begin{proof}
    This is a consequence of all of the preceeding work and the representer theorem.
\end{proof}


\section{Recovering the Fixed Effects}

While we are able to fit the model in terms of the kernel $r_\infty$, for scientific interpretation it is important to be able to uncover the fixed effects. This can be done by expanding the kernel $r_\infty$, yielding a representation of the solution in terms of the kernel $r_0$ and the actions along the fixed-effect directions. This is important because it allows us to recover the original, desired decomposition of the process $f(s)$ in terms of the original decomposition. 

\begin{prop}
    The solution to the optimization problem presented in Corollary \ref{cor:representer_cor} can be written as
    \begin{equation}
        f(s) = g(s)+X\beta
    \end{equation}
    for some function $g(s)\in \mathcal{H}_K(\Omega)$ and $\beta\in \mathbb{R}^p$.
\end{prop}
\begin{proof}
    We will prove this by constructing the function $g(s)$ and the coefficients $\beta$. By Eq. \eqref{eq:rinfty_kernelfn}, we can write Eq. \eqref{eq:repthm_minimizer} as
    \begin{equation}\label{eq:r0_and_a}
    \begin{split}
        f(s) =& \sum_{i=1}^d\alpha_ir_0(s,s_i) + \sum_{i=1}^d\alpha_i\frac{\gamma}{c}\sum_{j,k}S^{-1}_{jk}a_j(s)a_k(s_i)\\
        =&\sum_{i=1}^d\alpha_ir_0(s,s_i) + \sum_{j=1}^p a_j(s)\underbrace{\frac{\gamma}{c}\sum_{i,k}\alpha_iS_{jk}^{-1}a_k(s_i)}_{w_j}
        \end{split}
    \end{equation}
    Now since $a_j(s)=A^{-1}x_j(s)$ we can say that 
    \begin{equation}
        cKa_j(s)+\gamma a_j(s) = x_j(s) \implies  a_j(s) = \frac{1}{\gamma}\left(x_j(s)-cKa_j(s)\right)
    \end{equation}
    which is the sum of a term in $\SPAN{x_1,...,x_p}$ and a term in $\mathcal{H}_K(\Omega)$. Note that, since $a_j=A^{-1}x_j$ and $R_0=K(cK+\gamma I)^{-1}=KA^{-1}$, we can write that $Ka_j(s)=R_0x_j(s)$.
    
    Substituting this in to Eq. \eqref{eq:r0_and_a} we can see that
    \begin{equation}
        f(s) = \underbrace{\sum_{i=1}^d\alpha_ir_0(s,s_i) -\frac{c}{\gamma}\sum_{j=1}^pw_jR_0 x_j(s)}_{\in \mathcal{H}_K(\Omega)}+\underbrace{\sum_{j=1}^p \frac{w_j}{\gamma}x_j(s)}_{\in\SPAN{x_1,...,x_p}}
    \end{equation}
    with the first inclusion following from the fact that the spaces $\mathcal{H}_K$ and $\mathcal{H}_{R_0}$ are norm-equivalent and thus coincide as vector spaces. Now, setting $\beta_j=\frac{w_j}{\gamma}$ gives the desired form of the result.
\end{proof}
This result tells us not only that the function $f(s)$ has the exact form we hoped for when writing down the model, but it also gives us a way to recover the fixed-effect coefficients directly from the $\alpha_i$. So, once the problem has been solved, we are able to analyze the fixed effects directly without a separate estimation procedure. This allows for immediate scientific interpretation of the results. Applying, for example, the Bayesian procedure of \cite{pmlr-v70-walder17a}, wherein a Laplace approximation is made to the posterior, allows for additional analysis of the posterior distribution of these coefficients to assess model calibration and accuracy.



\section{Discussion}

In this paper, we have shown that it is possible to incorporate fixed effects inot permanental process models without losing the representer structure that makes these models so attractive to begin with. In the diffuse prior limit, we see that the penalty term and the equivalent kernel converge and describe the same RKHS, allowing for direct application of the representer theorem. Moreover, the result can be decomposed into a fixed effects term and a term corresponding to the RKHS of the original kernel, allowing the fixed effect coefficients to be recovered in closed for and making the model interpretable. 

This additionally has a similar interpretation to smoothing spline estimation. In the diffuse prior limit, this model treats $\SPAN{x_1,...,x_p}$ as directions in the limiting RKHS $\mathcal{H}_{R_\infty}$ along which estimation should be unpenalized. This appears again in the construction of the estimated function as $\hat{f}(s)=\hat{g}(s)+X\hat{\beta}$, similar to the result fit via smoothing spline methods based on similar ideas \cite{wahba1990spline}. 
Notably, this can be done with relatively minimal assumptions on the fixed effects $x_j(s)$, and also allows these covariates to enter the model in a relatively clean way, allowing for straightforward interpretation of the fixed effect coefficients. This can be contrasted, for example, with the models of \cite{kim2022fast} and \cite{kim2023survivalpermprocs}, which are similar but define $f$ to be a Gaussian process over the covariates. This leads to what the authors call a "triply intractable" problem that requires calculation of integrals over the covariate space and makes interpretation more difficult, even though the result still has the representer structure. 

Several directions exist for future work. In the immediate future, this needs to be accompanied by efficient computational schemes, for example based on the results of \cite{pmlr-v54-flaxman17a}, to make this method practical. Although the representer theorem reduces this to a finite optimization problem, it is still necessary to investigate efficient implementation as manipulation of large matrices may become inefficient. 

A second direction involves empirically studying the derived estimator of both the fixed and random effects components, and how varying the strength of wither affects the recovery of the other. It will also be important to understand how misspecification of the kernel matrix affects reconstruction of the fixed effects, and to compare the results derived from this method with more traditional log-Gaussian Cox models, permanental processes without fixed effects, and other candidate models. 

Finally, the present work considers a finite collection of fixed effects with a diffuse Gaussian prior. Extensions to more general effects, alternative priors, and other forms of partially penalized RKHS estimation may be possible using similar arguments. More broadly, the connection between the construction presented here and methods such as smoothing splines suggests that the same approach may be useful for incorporating unpenalized parametric components into other kernel-based point process models. 

\newpage
\bibliography{refs.bib}
\end{document}